\documentclass[11pt]{amsart}
\usepackage{float}
\usepackage{graphicx}
\usepackage{amssymb}
\theoremstyle{plain}
\newtheorem{thm}{Theorem}[section]
\newtheorem{lemma}[thm]{Lemma}
\newtheorem{cor}[thm]{Corollary}

\theoremstyle{definition}

\newtheorem{rmk}[thm]{Remark}

\def\dim{\mathop{\hbox {dim}}\nolimits}

\newcommand{\frg}{\mathfrak{g}}
\newcommand{\frh}{\mathfrak{h}}

\newcommand{\bbC}{\mathbb{C}}

\newcommand{\bbN}{\mathbb{N}}

\newcommand{\bbZ}{\mathbb{Z}}

\newcommand{\caM}{\mathcal{M}}
\newcommand{\caN}{\mathcal{N}}

\begin{document}

\title[Panyushev conjectures]
{Minuscule representations and
Panyushev conjectures}

\author{Chao-Ping Dong}

\address[Dong]
{Institute of Mathematics,  Hunan University,
Changsha 410082, China}
\email{chaoping@hnu.edu.cn}
\thanks{Dong is supported by NSFC grant 11571097 and the China Scholarship Council.}

\author{Guobiao Weng}

\address[Weng]
{School of Mathematical Sciences, Dalian University of Technology,
Liaoning 116024,  China} \email{gbweng@dlut.edu.cn}


\abstract{Recently, Panyushev raised five  conjectures concerning the structure of certain root posets arising from  $\bbZ$-gradings of simple Lie algebras.
This paper aims to provide proofs for four of them. Our study also links these posets with Kostant-Macdonald identity, minuscule representations, Stembridge's  ``$t=-1$ phenomenon", and the cyclic sieving phenomenon due to Reiner, Stanton and White.}
\endabstract



\maketitle


\section{Introduction}

The $\bbZ$-gradings of simple Lie algebras appear in different settings of Lie theory.
For instance, they occur naturally in the Jacobson-Morozov theorem of the orbit method due to Kirillov and Kostant (see \cite[Chapter 3]{CM}). They have also been used by Vinberg to construct algebraic groups \cite{Vin}. Recently, Panyushev raised five conjectures concerning the structure of certain root posets arising from $\bbZ$-gradings of simple Lie algebras \cite{P}. This paper aims to provide proofs to four of them, and our study will link these root posets with the following topics:
\begin{itemize}
\item[$\bullet$] Kostant-Macdonald identity \cite{Ko1, M};

\item[$\bullet$] Minuscule posets classified by Proctor \cite{Pr};

\item[$\bullet$] Stembridge's  ``$t=-1$ phenomenon" \cite{Stem};

\item[$\bullet$] Cyclic sieving phenomenon defined by Reiner, Stanton and White \cite{RSW}.
\end{itemize}

Now let us be more precise.
Let $\frg$ be a finite-dimensional simple Lie algebra over $\bbC$. Fix a Cartan subalgebra $\frh$ of $\frg$.
The associated root system is $\Delta=\Delta(\frg, \frh)\subseteq\frh^*$. Recall that a decomposition
\begin{equation}\label{grading}
\frg=\bigoplus_{i\in \bbZ}\frg(i)
\end{equation}
is  a \emph{$\bbZ$-grading} of $\frg$ if $[\frg(i), \frg(j)]\subseteq \frg(i+j)$ for any $i, j\in\bbZ$.
In particular, in such a case, $\frg(0)$ is a Lie subalgebra of $\frg$. Since each derivation of $\frg$ is inner, there exists $h_0\in\frg(0)$ such that $\frg(i)=\{x\in\frg\mid [h_0, x]=i x\}$. The element $h_0$ is said to be \emph{defining} for the grading \eqref{grading}. Without loss of generality, one may assume that $h_0\in\frh$. Then $\frh\subseteq\frg(0)$. Let $\Delta(i)$ be the set of roots in $\frg(i)$. Then we can
choose a set of positive roots $\Delta(0)^+$ for $\Delta(0)$ such that
$$
\Delta^+ :=\Delta(0)^+\sqcup \Delta(1)\sqcup \Delta(2)\sqcup \cdots
$$
is a set of positive
roots of $\Delta(\frg, \frh)$. Let $\Pi$ be the
corresponding simple roots, and put $\Pi(i)=\Delta(i)\cap
\Pi$. The grading \eqref{grading} is determined by $\Pi=\bigsqcup_{i\geq 0} \Pi(i)$.
When $|\Pi(1)|=1$ and $\Pi(i)$ vanishes for $i\geq 2$, we say
the $\bbZ$-grading \eqref{grading} is \emph{$1$-standard}.  In this case $\Delta(1)$ becomes
\begin{equation}
[\alpha_i]:=\{\alpha\in\Delta^+ \mid   [\alpha: \alpha_i]=1\}
\end{equation}
Here $\alpha_i$ is a simple root, and $[\alpha: \alpha_i]$ is the
coefficient of $\alpha_i$ in $\alpha$.
Note that $[\alpha_i]$
inherits a poset
structure from the usual one of $\Delta^+$: let $\alpha$
and $\beta$ be two roots of $[\alpha_i]$, then $\alpha\leq\beta$ if
and only if $\beta-\alpha$ is a nonnegative integer combination of
simple roots.  The root poset $[\alpha_i]$ is the core object of this paper. We will use results of
Ringel \cite{R} to analyze it in Section 4.

In \cite{P}, Panyushev raised several beautiful conjectures concerning the
$\mathcal{M}$-polynomial, $\mathcal{N}$-polynomial and the reverse
operator in $\Delta(1)$. Before stating them, let us
prepare a bit more notation.  Recall that a subset $I$ of a finite poset $(P, \leq)$ is a \emph{lower} (resp., \emph{upper}) \emph{ideal} if $x\leq y$ in $P$ and $y\in I$ (resp. $x\in I$) implies that $x\in I$ (resp. $y\in I$). Let $J(P)$ be the lower ideals of $P$, partially ordered by inclusion. A subset $A$ of $P$ is an \emph{antichain} if its elements are mutually incomparable. Note that the following maps give bijections between lower ideals, upper ideals and antichains of $P$:
\begin{equation}
I\mapsto P\setminus I \mapsto \min(P\setminus I).
\end{equation}
Denote by $\mathcal{M}_{P}(t)$ the generating function of lower
ideals of $P$. That is, $\mathcal{M}_{P}(t):=\sum_{I} t^{|I|}$,
where $I$ runs over the lower ideals of $P$. Denote by
$\mathcal{N}_{P}(t)$ the generating function of antichains of $P$.
That is, $\mathcal{N}_{P}(t):=\sum_{A} t^{|A|}$, where $A$ runs over
the antichains of $P$.

As on p.~244 of Stanley \cite{St}, a finite poset $P$ is said to be
\emph{graded} if every maximal chain in $P$ has the same length. In
this case, there is a unique rank function $r$ from $P$ to the
positive integers $\mathbb{P}$ such that all the minimal elements
have rank $1$, and $r(x)=r(y)+1$ if $x$ covers $y$. The model for our concern is
$\Delta(1)$, where the height function $\rm{ht}$ gives the rank.

Now Conjecture 5.1 of \cite{P} is stated as follows.

\medskip
\noindent \textbf{Panyushev's $\caM$-polynomial conjecture.}
\emph{For any $\bbZ$-grading of $\frg$, we have
\begin{equation}\label{KM}
\mathcal{M}_{\Delta(1)}(t)=\prod_{\gamma\in\Delta(1)}
\frac{1-t^{\rm{ht}(\gamma)+1}}{1-t^{\rm{ht}(\gamma)}}.
\end{equation}}

\medskip

The RHS of \eqref{KM} traces back to the celebrated
Kostant-Macdonald identity (see \cite{AC}, \cite{Ko1} and  Corollary 2.5 of
\cite{M}) saying that
$$
\sum_{w\in W}t^{l(w)}=\prod_{\gamma\in\Delta^+}\frac{1-t^{\rm{ht}(\gamma)+1}}{1-t^{\rm{ht}(\gamma)}}.
$$
Here $W$ is the Weyl group associated with $\Delta^+$, and $l(\cdot)$ is the length function.

When the grading \eqref{grading} is \emph{abelian}
(i.e., when $\Delta(i)$ vanishes for $i\geq 2$), the poset
$[\alpha_i^{\vee}]$ in the dual root system $\Delta^{\vee}$ is \emph{minuscule} in the sense of
Proctor \cite{Pr}, see Section 3 for more details. According to Exercise 3.170 of Stanley \cite{St}, we call a finite
graded poset $P$ \emph{pleasant} if \eqref{KM}, with $\rm{ht}$ replaced by the rank function $r$, holds for it. Thus
Panyshev's $\caM$-polynomial conjecture asserts that each
$\Delta(1)$ is pleasant.

The first aim of this paper is to remark that Panyushev's
$\caM$-polynomial conjecture follows from Proctor's Theorem (see
Theorem \ref{thm-Proctor}) plus some additional effort. The key
observation is that for all but seven exceptions (see Section 4.10)
these $[\alpha_i]$ bear the pattern
\begin{equation}\label{pattern}
[\alpha_i]\cong [k]\times P,
\end{equation}
where $k$ is a positive integer, $[k]$ denotes the totally ordered
set $\{1, 2, \cdots, k\}$, and $P$ is a connected minuscule poset
classified in Theorem \ref{thm-Proctor}. As a consequence, we obtain the following.

\begin{thm}\label{thm-M-poly-main}
Panyushev's $\caM$-polynomial conjecture is true.
\end{thm}

By definition, the number $\caM_{\Delta(1)}(1)$ counts the lower ideals of $\Delta(1)$. Thus we have

\begin{cor}\label{cor-thm-M-poly-main}
For any $\bbZ$-grading of $\frg$,
\begin{equation}\label{number-antichains}
\caM_{\Delta(1)}(1)=|J(\Delta(1))|=\prod_{\gamma\in\Delta(1)}\frac{\rm{ht}(\gamma)+1}{\rm{ht}(\gamma)}.
\end{equation}
\end{cor}

Let $E$ (resp. $F$) be the \emph{multi-set} of the even (resp. odd) heights of $\Delta(1)$. By Theorem \ref{thm-M-poly-main}, we have
\begin{equation}\label{M-1}
\caM_{\Delta(1)}(-1)=
\begin{cases}
\prod_{f\in F}(f+1)/\prod_{e\in E} e &
\mbox{ if } |E|=|F|, \\ 0&
\mbox{ otherwise.}
\end{cases}
\end{equation}
It is far from being evident that the number $\caM_{\Delta(1)}(-1)$ counts certain lower
ideals of $\Delta(1)$ enjoying nice symmetries. Indeed, suppose that $c: P\to P$ is an order-reversing involution on the finite poset $(P, \leq)$. After Stembridge \cite{St}, we
call the triple $(P, \leq, c)$ a \emph{complemented} poset. In such
a case, for any $I\in J(P)$, put $I^c:=P\setminus \{c(x)\mid x\in
I\}$. Then $I\mapsto I^c$ is an order-reversing involution on
$J(P)$. This makes $J(P)$ into a complemented poset as well, for
which we denote by $(J(P), \subseteq, c)$, or simply by $(J(P), c)$.
We call a lower ideal $I\in J(P)$
\emph{self-complementary} if $I=I^c$. In our situation, let $w_0^i$
be the longest element of the Weyl group of $\frg(0)$ coming from
the $1$-standard $\bbZ$-grading such that $\Pi(1)=\{\alpha_i\}$.
Note that $w_0^i(\Delta(1))=\Delta(1)$, and the
$w_0^i$ action on $\Delta(1)=[\alpha_i]$ makes it into a
complemented poset, for which we denote by $([\alpha_i], w_0^i)$. We
denote the  corresponding complemented poset structure on
$J([\alpha_i])$ by $J([\alpha_i], w_0^i)$.

In Lemma \ref{lemma-uniform-transfer}, we shall transfer the order-reversing involution on
 each minuscule weight lattice coming from the $w_0$ action to the corresponding minuscule poset $\Delta(1)$. Here $w_0$ is the longest element of the Weyl group $W(\frg, \frh)$. Then we will build up
further links between the pattern \eqref{pattern} and the minuscule
representations. This makes Stembridge's  ``$t=-1$ phenomenon" (see Theorem 4.1 of \cite{Stem})
 applicable, and leads us to the following.

\begin{thm}\label{thm-M-1-main} Conjecture 5.2 of
\cite{P} is true: For any $1$-standard
$\bbZ$-grading $\frg=\bigoplus_{j\in \bbZ}\frg(j)$ of $\frg$ such that $\Pi(1)=\{\alpha_i\}$, the number $\caM_{\Delta(1)}(-1)$ counts the self-complementary lower ideals in $J([\alpha_i], w_0^i)$.
\end{thm}

Originally, Conjecture 5.2 of \cite{P} is stated in terms of upper
ideals. In Lemma
\ref{lemma-fixed-pt}, we will show that a lower ideal $I$ of $\Delta(1)$ is
self-complementary if and only if the upper ideal
$\Delta(1)\setminus I$ is self-complementary. Thus we can interpret the above
theorem in terms of upper ideals instead.

It is interesting to ask that when does the number
$\caM_{\Delta(1)}(-1)$ vanish? A direct answer using
\eqref{M-1} is that this happens if and only if $|E|\neq |F|$. A deeper characterization is found as follows.

\begin{thm}\label{thm-fixed-point-main} Let $\frg=\bigoplus_{j\in \bbZ}\frg(j)$ be any $1$-standard
$\bbZ$-grading of $\frg$. Then $\caM_{[\alpha_i]}$ vanishes at $-1$
if and only if the $w_0^i$ action on $[\alpha_i]$ has fixed
point(s).
\end{thm}

Let $P_i$ be the set
of elements in the finite graded poset $P$ with rank $i$. The sets
$P_i$ are said to be the \emph{rank levels} of $P$. Suppose that
$P=\bigsqcup_{i=1}^{d} P_i$. Then $P$ is called \emph{Sperner} if
the largest size of an antichain equals $\max\{|P_i|, 1\leq i\leq
d\}$.  Conjecture 5.12 of \cite{P} is stated as follows.

\medskip
\noindent \textbf{Panyushev's $\caN$-polynomial conjecture.}
\emph{Let $\frg=\bigoplus_{i\in \bbZ}\frg(i)$ be any $1$-standard
$\bbZ$-grading of $\frg$. Then $\caN_{\Delta(1)}$ is palindromic if
and only if $\Delta(1)$ has  a unique rank level of maximal size.}
\medskip

The second theme of this
paper is to confirm the above conjecture.

\begin{thm}\label{thm-N-poly-main} Let $\frg=\bigoplus_{i\in \bbZ}\frg(i)$ be any $1$-standard
$\bbZ$-grading of $\frg$. The following are equivalent:
\begin{itemize}
\item[a)] $\mathcal{N}_{\Delta(1)}$ is palindromic;
\item[b)] $\mathcal{N}_{\Delta(1)}$ is monic;
\item[c)] $\Delta(1)$ has a unique antichain of maximal size;
\item[d)] $\Delta(1)$ has a unique rank level of maximal size.
\end{itemize}
In particular, Panyushev's $\caN$-polynomial conjecture is true.
\end{thm}

We
collect the antichains of $P$ as $\mathrm{An}(P)$. For any $x\in P$,
let $I_{\leq x}=\{y\in P\mid y\leq x\}$. Given an antichain $A$ of
$P$, let $I(A)=\bigcup_{a\in A} I_{\leq a}$. The \emph{reverse
operator} $\mathfrak{X}$ is defined by $\mathfrak{X}(A)=\min
(P\setminus I(A))$.
Since antichains of $P$ are in bijection with
lower (resp. upper) ideals of $P$, the reverse operator acts on
lower (resp. upper) ideals of $P$ as well. Note that the current
$\mathfrak{X}$ is inverse to the reverse operator
$\mathfrak{X}^{\prime}$ in Definition 1 of \cite{P}, see Lemma
\ref{lemma-inverse-reverse-operator}. Hence replacing
$\mathfrak{X}^{\prime}$ by $\mathfrak{X}$ does not affect our
forthcoming discussion on orbits.
When $P$ is a root poset, we call $\mathfrak{X}$ the \emph{Panyushev operator} and call a  $\mathfrak{X}$-orbit a \emph{Panyushev orbit}. The third theme of this paper is the structure of Panyushev orbits of $\Delta(1)$.

The $\bbZ$-grading \eqref{grading} is \emph{extra-special} if
\begin{equation}\label{extra-special}
\frg=\frg(-2)\oplus \frg(-1) \oplus \frg(0) \oplus \frg(1)
\oplus \frg(2) \mbox{  and  }\dim\frg(2)=1,
\end{equation}
Up to conjugation, any simple Lie algebra $\frg$ has a unique extra-special $\bbZ$-grading. Without loss of generality, we assume that $\Delta(2)=\{\theta\}$, where $\theta$ is the highest root of $\Delta^+$.
Namely, we may assume that the grading \eqref{extra-special} is defined by the element $\theta^{\vee}$, the dual root of $\theta$.

In such a case, we have
\begin{equation}\label{Delta-one}
\Delta(1)=\{\alpha\in\Delta^+\mid (\alpha, \theta^{\vee})=1\}.
\end{equation}
Recall that $h:=\mathrm{ht}(\theta)+1$ is the \emph{Coxeter number}
of $\Delta$. Let $h^*$ be the \emph{dual Coxeter number }of
$\Delta$. That is, $h^*$ is  one plus the height of
$\theta^{\vee}$ in $\Delta^{\vee}$. As noted on p.~1203 of \cite{P},
we have $|\Delta(1)|=2h^*-4$. We call a lower (resp. upper) ideal
$I$ of $\Delta(1)$ \emph{Lagrangian} if $|I|=h^*-2$. Write
$\Delta_l$ (resp. $\Pi_l$) for the set of \emph{all} (resp.
\emph{simple}) \emph{long} roots. In the simply-laced cases, all
roots are assumed to be both long and short. Note that $\theta$ is
always long, while $\theta^{\vee}$ is always short.

\begin{thm}\label{thm-main-reverse-operator-orbit} Conjecture 5.11 of \cite{P} is true: In any extra-special
$\bbZ$-grading of $\frg$, the number of
Panyushev orbits equals $|\Pi_l|$, and each orbit
is of size $h-1$. Furthermore, if $h$ is even (which only excludes the case $A_{2k}$ where $h=2k+1$), then each
Panyushev orbit contains a unique Lagrangian lower
ideal.
\end{thm}

Originally, Conjecture 5.11 of \cite{P} was stated in terms of upper
ideals and  $\mathfrak{X}^{\prime}$. Equivalently, we can phrase it using lower ideals and the Panyushev operator
$\mathfrak{X}$.

The \emph{cyclic sieving phenomenon (CSP)} was defined by Reiner, Stanton and White \cite{RSW} as follows: let $X$ be a finite set, let $X(t)$ be a polynomial in $t$ whose coefficients are nonnegative integers and let $C=\langle c\rangle$ be a cyclic group of order $n$ acting on $X$.
The triple $(X, X(t), C)$ exhibits the CSP if
\begin{equation}\label{CSP-1}
X(t)\big|_{t=\zeta^k}=\big|\{x\in X \mid c^k(x)=x\}\big|,
\end{equation}
where $\zeta$ is a primitive $n$-th root of unity. Let
 \begin{equation}\label{CSP-2}
X(t)\equiv a_0 + a_1 t+\cdots +a_{n-1}t^{n-1} \mod (t^n-1).
\end{equation}
By Proposition 2.1 of  \cite{RSW}, an equivalent way to define the CSP is to say that
$a_i$ equals the number of $C$-orbits in $X$ whose stabilizer order divides $i$.
The following result is a slight extension of the main theorems of Rush and Shi \cite{RS}.

\begin{thm}\label{thm-main-cyclic-sieving} Let $\frg=\bigoplus_{i\in \bbZ}\frg(i)$ be any $1$-standard
$\bbZ$-grading of $\frg$. Then the triple $(\Delta(1)$, $\caM_{\Delta(1)}(t)$, $\langle\mathfrak{X}_{\Delta(1)}\rangle)$
exhibits the CSP.
\end{thm}

We adopt computer verifications via
\texttt{Mathematica} in the following cases: the seven posets
violating the pattern \eqref{pattern} for Theorems
\ref{thm-M-poly-main}, \ref{thm-M-1-main},
\ref{thm-N-poly-main}, and \ref{thm-main-cyclic-sieving}; the exceptional Lie algebras for Theorem
\ref{thm-main-reverse-operator-orbit}. The program files are
available from the first named author.

The paper is organized as follows: We prepare some preliminaries in
Section 2, and recall Proctor's Theorem in Section 3. Then we
analyze the structure of the posets $[\alpha_i]$ in Section 4, and
show Theorem \ref{thm-M-poly-main} in Section 5.  We make
Stembridge's theorem applicable, and prove Theorems \ref{thm-M-1-main} and
\ref{thm-fixed-point-main} in Section 6. We deduce
some results on $\caN$-polynomials and verify
Theorem \ref{thm-N-poly-main} in Section 7. Finally,
Theorems \ref{thm-main-reverse-operator-orbit} and \ref{thm-main-cyclic-sieving}
are obtained in Section 8.

\section{Preliminary results}

Throughout this paper, $\bbN :=\{0, 1, 2, \dots\}$, and
$\mathbb{P}:=\{1, 2, \dots\}$. For each $k\in\mathbb{P}$, the poset
$[k]$ is equipped with the order-reversing involution $c$ such that
$c(i)=k+1-i$. We denote $J(J(P))$ and $J(J(J(P)))$ by $J^2(P)$ and
$J^3(P)$, respectively.

No let us collect some preliminary results. Let
$(P_i,\leq), i=1, 2$ be two finite posets. One can define a poset
structure on $P_1\times P_2$ by setting $(u_1, v_1)\leq (u_2, v_2)$
if and only if $u_1\leq u_2$ in $P_1$ and $v_1\leq v_2$ in $P_2$. We
simply denote the resulting poset by $P_1 \times P_2$. The following
lemma gives all lower ideals of $P_1\times P_2$.

\begin{lemma}\label{lemma-P1P2}
Let $P_1, P_2$ be two finite posets. Let $S$ be a subset of
$P_1\times P_2$. For each $u\in P_1$, put $S_u=\{v \in
P_2|(u,v)\in S\}$. Then $S$ is a lower ideal of $P_1\times
P_2$ if and only if $S_u$ is a lower ideal of $P_2$ for each
$u\in P_1$, and that $S_{u_1}\supseteq S_{u_2} $ whenever
$u_1\leq u_2$ in $P_1$.
\end{lemma}
\begin{proof}
It suffices to prove the sufficiency. Given $(u,v)\in S$, take any
$(x,y)\in P_1\times P_2$ such that $(x, y)\leq (u,v)$, then $x\leq
u$ and $y\leq v$. Firstly, we have $y\in S_u$ since $S_u$ is a lower
ideal of $P_2$ and $v\in S_u$. Secondly, since $x\leq u$, we have
$S_x\supseteq S_u$. Hence $y\in S_x$, i.e., $(x, y)\in S$. This
proves that $S$ is a lower ideal of $P_1\times P_2$.
\end{proof}

As a direct consequence, we have the following well-known result
describing the lower ideals of $[n]\times P$.

\begin{lemma}\label{lemma-ideals-CnP}
Let $P$ be a finite poset. Let $I$ be a subset of $[m]\times
P$. For $1\leq i\leq m$, denote $I_i=\{a\in P\mid (i, a)\in I\}$.
Then $I$ is a lower ideal of $[m]\times P$ if and only if each $I_i$
is a lower ideal of $P$, and $I_m\subseteq I_{m-1}\subseteq \cdots \subseteq I_{1}$.
\end{lemma}

The following
lemma describes the antichains of $P_1\times P_2$.

\begin{lemma}\label{lemma-antichain-P1P2}
Let $P_1, P_2$ be two finite posets. Let $A$ be a subset of
$P_1\times P_2$. For each $u\in P_1$, put $A_u=\{v \in
P_2|(u,v)\in A\}$. Then $A$ is an antichain of $P_1\times
P_2$ if and only if $A_u$ is an antichain of $P_2$ for each
$u\in P_1$, and that $A_{u_1} \subseteq P_2 \setminus I(A_{u_2}) $ whenever
$u_1\leq u_2$ in $P_1$.
\end{lemma}
\begin{proof}
Note that $A=\bigsqcup_{u\in P_1} \{(u, v)\mid  v\in A_{u}\}$. Then use the definition of antichain.
\end{proof}

As a direct corollary, we have the following.

\begin{lemma}\label{lemma-antichain-CnP}
Let $P$ be a finite poset. Let $A$ be a subset of $[m]\times P$. For
$1\leq i\leq m$, denote $A_i=\{a\in P\mid (i, a)\in A\}$. Then $A$
is an antichain of $[m]\times P$ if and only if each $A_i$ is an
antichain of $P$, and $A_i\subseteq P\setminus I(A_{j})$ for
$1\leq i <j \leq m$.
\end{lemma}

Now let us compare the two reverse operators. Let $(P, \leq)$ be a
finite poset. For any $x\in P$, let $I_{\geq x}=\{y\in P\mid x\leq
y\}$. For any antichain $A$ of $P$, put $I_{+}(A)=\bigcup_{a\in A}
I_{\geq a}$. Recall that in Definition 1 of \cite{P}, the reverse
operator $\mathfrak{X}^{\prime}$ is given by
$\mathfrak{X}^{\prime}(A)=\max (P\setminus I_{+}(A))$.

\begin{lemma}\label{lemma-inverse-reverse-operator}
The operators $\mathfrak{X}$ and  $\mathfrak{X}^{\prime}$ are
inverse to each other.
\end{lemma}
\begin{proof}
Take any antichain $A$ of $P$, note that
$$I_{+}(\min(P\setminus
I(A)))=P\setminus I(A)\mbox{ and } I(\max(P\setminus
I_{+}(A)))=P\setminus I_{+}(A).
$$
Then the lemma follows.
\end{proof}

Suppose that
$P=\bigsqcup_{j=1}^{d} P_j$ is the decomposition of a finite graded poset $P$ into rank levels. Let $P_0$ be the empty set $\emptyset$.
Put $L_i=\bigsqcup_{j=1}^{i} P_j$ for $1\leq i\leq d$, and let $L_0$
be the empty set. We call those $L_i$ \emph{full rank} lower
ideals. Recall that the reverse operator
acts on lower ideals as well. For instance,
$\mathfrak{X}_P(L_i)=L_{i+1}$, $0\leq i<d$ and $\mathfrak{X}_P(L_d)=L_{0}$.

Let $\mathfrak{X}$  be the reverse operator on $[m]\times P$. In
view of Lemma \ref{lemma-ideals-CnP}, we \emph{identify}  a general
lower ideal of $[m]\times P$ with $(I_1, \dots, I_m)$, where each
$I_i\in J(P)$ and $I_m\subseteq  \cdots \subseteq I_{1}$. We say
that the lower ideal $(I_1, \dots, I_m)$ is \emph{full rank} if
each $I_i$ is full rank in $P$.  Let $\mathcal{O}(I_1, \dots, I_m)$ be
the $\mathfrak{X}_{[m]\times P}$-orbit of $(I_1, \dots, I_m)$. We prepare the
following.

\begin{lemma}\label{lemma-operator-ideals-CmP}
Keep the notation as above.  Then
for any $n_0\in \bbN$, $n_i\in\mathbb{P}$ ($1\leq i\leq s$) such that $\sum_{i=0}^{s} n_i =m$, we have
\begin{equation}\label{rank-level}
\mathfrak{X}_{[m]\times P}(L_d^{n_0}, L_{i_1}^{n_1}, \dots, L_{i_s}^{n_s})=
(L_{i_1+1}^{n_0+1}, L_{i_2+1}^{n_1}, \dots, L_{i_s+1}^{n_{s-1}}, L_0^{n_s-1}),
\end{equation}
where $0\leq i_s<\cdots <i_1<d$, $L_d^{n_0}$ denotes $n_0$ copies of $L_d$ and so on.

\end{lemma}
\begin{proof}
Note that under the above assumptions, $(L_d^{n_0}, L_{i_1}^{n_1}, \dots, L_{i_s}^{n_s})$ is a lower ideal of $[m]\times P$ in view of Lemma \ref{lemma-ideals-CnP}. Then analyzing the minimal elements of $([m]\times P)\setminus (L_d^{n_0}, L_{i_1}^{n_1}, \dots, L_{i_s}^{n_s})$ leads one to \eqref{rank-level}.

\end{proof}

\begin{lemma}\label{lemma-operator-types}
Let $(I_1, \dots, I_m)$ be an arbitrary lower ideal of $[m]\times P$.
Then $(I_1, \dots, I_m)$ is full rank if and only if each lower ideal
in the orbit $\mathcal{O}(I_1, \dots, I_m)$ is full rank.
\end{lemma}
\begin{proof}
Use Lemma \ref{lemma-operator-ideals-CmP}.
\end{proof}

Due to the above lemma, we say the $\mathfrak{X}_{[m]\times P}$-orbit $\mathcal{O}(I_1, \dots, I_m)$
is of \emph{type I} if $(I_1, \cdots, I_m)$ is full rank, otherwise we say $\mathcal{O}(I_1, \dots, I_m)$
is of \emph{type II}.

\begin{figure}[H]
\centering \scalebox{0.24}{\includegraphics{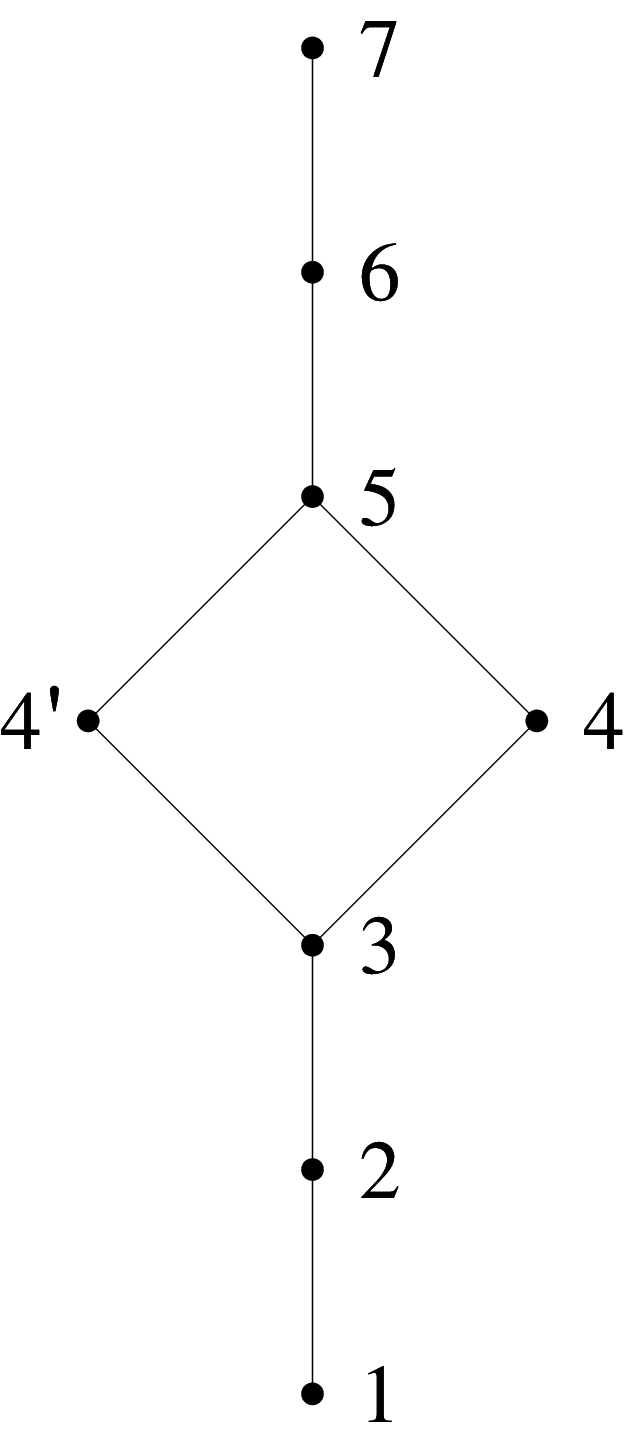}}
\caption{The labeled Hasse diagram of $K_3$}
\end{figure}

For any $n\geq 2$, let $K_{n-1}=[n-1]\oplus([1]\sqcup [1])\oplus
[n-1]$ (the ordinal sum, see p.~246 of \cite{St}).  We label the
elements of $K_{n-1}$ by $1$, $2$, $\dots$, $n-1$, $n$,
$n^{\prime}$, $n+1$, $\dots$, $2n-2$, $2n-1$. Fig.~1 illustrates
the labeling for $K_3$. Note that $L_i$ ($0\leq i\leq 2n-1$) are all
the full rank lower ideals. For instance, we have $L_{n}=\{1, 2,
\dots, n, n^{\prime}\}$. Moreover, we put $I_{n}=\{1,  \dots, n-1,
n\}$ and $I_{n^{\prime}}=\{1,  \dots, n-1, n^{\prime}\}$. The
following lemma will be helpful in analyzing the
$\mathfrak{X}_{[m]\times K_{n-1}}$-orbits of type II.

\begin{lemma}\label{lemma-operator-ideals-CmK}
Fix $n_0\in \bbN$, $n_i\in\mathbb{P}$ ($1\leq i\leq s$),
$m_j\in\mathbb{P}$ ($0\leq j\leq t$) such that $\sum_{i=0}^{s} n_i +
\sum_{j=0}^{t} m_j=m$. Take any $0\leq j_t< \cdots<j_1<n\leq
i_s<\cdots <i_1<2n-1$, we have
\begin{align*}
\mathfrak{X}_{[m]\times K_{n-1}}&(L_{2n-1}^{n_0}, L_{i_1}^{n_1}, \dots, L_{i_s}^{n_s}, I_n^{m_0}, L_{j_1}^{m_1}, \dots, L_{j_t}^{m_t})=\\
&\begin{cases}
( L_{i_1+1}^{n_0+1}, L_{i_2+1}^{n_1}, \dots, L_{i_s+1}^{n_{s-1}}, I_{n^{\prime}}^{n_s}, L_{j_1+1}^{m_0},
L_{j_2+1}^{m_1}, \dots, L_{j_t+1}^{m_{t-1}}, L_0^{m_t-1} ) & \mbox { if } j_1 < n-1;\\
( L_{i_1+1}^{n_0+1}, L_{i_2+1}^{n_1}, \dots, L_{i_s+1}^{n_{s-1}}, L_{n}^{n_s}, I_n^{m_0},
\, \, \, \, L_{j_2+1}^{m_1}, \dots, L_{j_t+1}^{m_{t-1}}, L_0^{m_t-1} )& \mbox { if } j_1 = n-1.
\end{cases}
\end{align*}
\end{lemma}
\begin{proof}
Analyzing the minimal elements of $$([m]\times K_{n-1})\setminus (L_{2n-1}^{n_0}, L_{i_1}^{n_1}, \dots, L_{i_s}^{n_s}, I_n^{m_0}, L_{j_1}^{m_1}, \dots, L_{j_t}^{m_t})$$ leads one to the desired expression.
\end{proof}

\section{Proctor's Theorem}
In this section, we will recall minuscule representations, minuscule
posets, and a theorem of Proctor. We continue to denote by $\frg$ a
finite-dimensional simple Lie algebra over $\bbC$ with rank $n$. Let
$V_{\lambda}$ be a finite-dimensional irreducible $\frg$-module with
highest weight $\lambda$. Denote by $\Lambda_{\lambda}$ the
multi-set of weights in $V_{\lambda}$. One says that $V_{\lambda}$
(and hence also ${\lambda}$) is \emph{minuscule} if the action of
$W$ on $\Lambda_{\lambda}$ is transitive. By Exercise VI.1.24 of
Bourbaki \cite{B}, a minuscule weight $\lambda$ must be a
fundamental weight. However, the converse is not true. We refer the reader to the appendix of
\cite{Stem} for  a complete list of minuscule weights.

Now let $V_{\varpi_i}$ be a minuscule representation, where
$\varpi_i$ is the fundamental weight corresponding to the $i$-th
simple root $\alpha_i\in\Pi$. Namely, for any $1\leq j \leq n$,
\begin{equation}
( \varpi_i, \alpha_j^{\vee} )=\delta_{ij},
\end{equation}
where $\alpha_j^{\vee}=2\alpha_j /\|\alpha_j\|^2$. Then by
Proposition 4.1 of \cite{P}, one knows that the poset
$\Lambda_{\varpi_i}$ is a distributive lattice. Thus by Theorem
3.4.1 of \cite{St}, there is a (unique) poset $P_{\varpi_i}$ such
that $\Lambda_{\varpi_i}\cong J(P_{\varpi_i})$. Indeed, we point out
that
\begin{equation}\label{minu-real}
P_{\varpi_i}\cong [\alpha_i^{\vee}] \mbox{ in } (\Delta^{\vee})^+,
\end{equation}
where $\Delta^{\vee}$ is the root system dual to $\Delta$. Moreover,
these $P_{\varpi_i}$ are exactly the \emph{minuscule posets} in the
sense of \cite{P}.

Let us recall from Exercise 3.172 of \cite{St} that a finite graded
poset $P=\{t_1, \dots, t_p\}$ is \emph{Gaussian} if there exists
positive integers $h_1, \dots, h_p>0$ such that for each $m\in
\bbN$,
\begin{equation}\label{Gaussian}
\mathcal{M}_{[m]\times
P}(t)=\prod_{i=1}^{p}\frac{1-t^{m+h_i}}{1-t^{h_i}}.
\end{equation}

Now let us state Proctor's theorem, which is a combination of
Proposition 4.2 and Theorem 6 of \cite{Pr}.

\begin{thm}\label{thm-Proctor} \emph{(\textbf{Proctor})}
The connected minuscule posets are classified as below: $[n]\times
[m]$, for all $m, n\in \mathbb{P}$; $K_r:=[r]\oplus([1]\sqcup [1])\oplus
[r]$ (the ordinal sum, see p.~246 of \cite{St}), for all $r\in
\mathbb{P}$; $H_r:=J([2]\times [r-1])$, for all $r\in \mathbb{P}$;
$J^2([2]\times [3])$ and $J^3([2]\times [3])$. Moreover, each
minuscule  poset is Gaussian.
\end{thm}
\begin{rmk}\label{rmk-thm-Proctor}
By Exercise 3.172 of \cite{St}, $P$ is Guassian if and only if
$P\times [m]$ is pleasant for each $m\in \mathbb{P}$.
\end{rmk}

For reader's convenience, we present the Hasse diagrams of
$J^2([2]\times [3])$ and $J^3([2]\times [3])$ in  Fig.~2.


\begin{figure}[]
\centering \scalebox{0.3}{\includegraphics{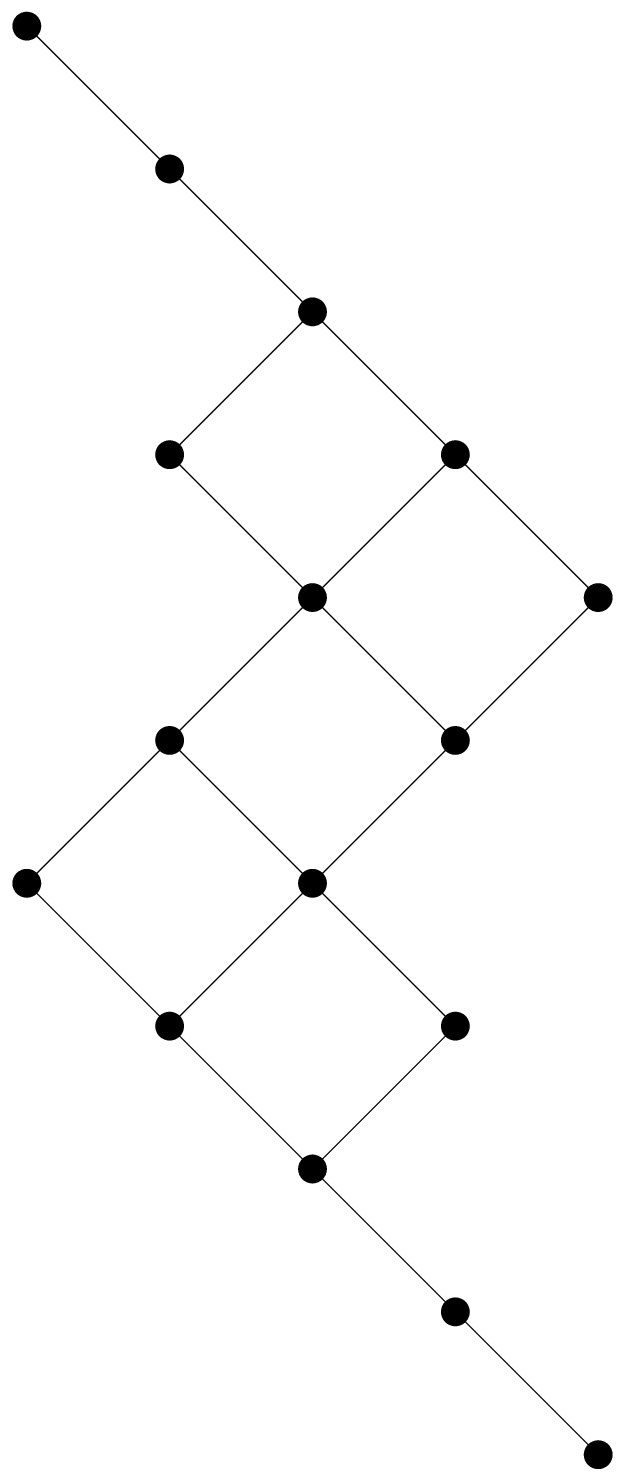}}
\qquad\quad\quad\qquad \scalebox{0.25}{\includegraphics{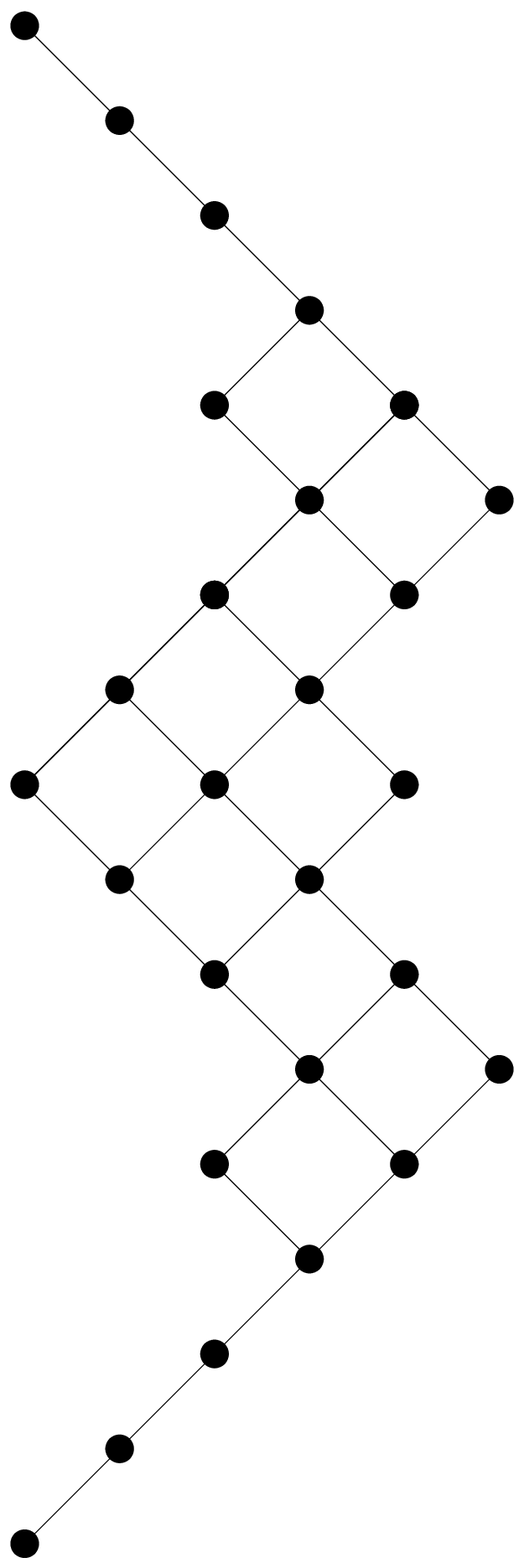}}

\caption{The Hasse diagrams of $J^2([2]\times [3])$ (left) and
$J^3([2]\times [3])$ (right)}
\end{figure}

\section{The structure of $\Delta(1)$}

This section is devoted to understanding the structure of
$\Delta(1)=[\alpha_i]$ in a general $1$-standard $\bbZ$-grading.
Ringel's paper \cite{R} is very helpful on this aspect. The main
point is to demonstrate that most of these $[\alpha_i]$ observe the
pattern \eqref{pattern}. Those $[\alpha_i]$ violating the pattern
\eqref{pattern} are listed in the final subsection.

Although our
discussion below is case-by-case, the underlying method is the same. Indeed, let $X_n$
be the type of $\frg$.
If $\alpha_i$ is not a branching point, then there are two (sub) connected components
in the Dynkin diagram of $X_n$ containing $\alpha_i$ as an ending point: one is $A_{k}$, and the other is $Y_{n-k+1}$. Here $k=1$ if $\alpha_i$ itself is an ending point in $X_n$. Then $A_{k}$ produces $[k]$, while the minuscule poset $P$ is related to $Y_{n-k+1}$. Actually,
if in $Y_{n-k+1}$ the fundamental weight corresponding to $\alpha_i$
is minuscule, then $P$ is just $[\alpha_i]$ in $Y_{n-k+1}$.  If
$\alpha_i$ is a branching point, then there are three (sub) connected components: $A_{2}$, $A_{r+1}$ and
$A_{s+1}$, where $r+s=n-2$, and we have $[\alpha_i]\cong [2]\times
[r+1] \times [s+1]$.

\subsection{$A_n$}
We fix  $\alpha_i=e_i-e_{i+1}$, $1\leq i\leq n$, then
$$
[\alpha_i]\cong [i]\times [n+1-i], \quad 1\leq i\leq n.
$$

\subsection{$B_n$}

We fix $\alpha_i=e_i-e_{i+1}$, $1\leq i\leq n-1$, and $\alpha_n=e_n$
as the simple roots. Then
$$
[\alpha_i] \cong [i]\times [2n+1-2i], \quad 1\leq i\leq n.
$$

\subsection{$C_n$}

We fix $\alpha_i=e_i-e_{i+1}$, $1\leq i\leq n-1$, and
$\alpha_n=2e_n$ as the simple roots. Then $[\alpha_n]\cong H_n$, and
$$
[\alpha_i] \cong [i]\times [2n-2i], \quad 1\leq i\leq n-1.
$$

\subsection{$D_n$}
We fix $\alpha_i=e_i-e_{i+1}$, $1\leq i\leq n-1$, and
$\alpha_n=e_{n-1}+e_n$ as the simple roots. Then
$[\alpha_{n-1}]\cong [\alpha_n]\cong H_{n-1}$, and
\begin{equation}\label{}
[\alpha_i]\cong [i] \times K_{n-i-1}, \quad 1\leq i\leq n-2.
\end{equation}

\subsection{$G_2$}
Let $\alpha_1$ be the short simple root, and let $\alpha_2$ be the
long simple root. Then $[\alpha_1]\cong [2]$ and $[\alpha_2]\cong
[4]$.

\subsection{$F_4$}
The Dynkin diagram is as follows, where the arrow points from long
roots to short roots.

\begin{figure}[H]
\centering \scalebox{0.6}{\includegraphics{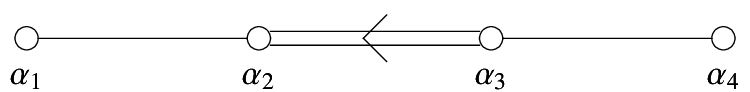}}
\end{figure}

Then $[\alpha_1]\cong K_3$, $[\alpha_2]\cong [2]\times [3]$,
$[\alpha_3]\cong [2]\times K_2$.

\subsection{$E_6$}
The Dynkin diagram of $E_6$ is given below. Note that our labeling
of the simple roots agrees with p.~687 of \cite{K}, while differs
from that of \cite{P}.
\begin{figure}[H]
\centering \scalebox{0.6}{\includegraphics{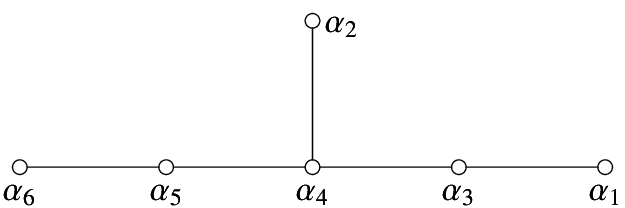}}
\end{figure}
Then $[\alpha_1]\cong [\alpha_6]\cong J^2([2]\times [3])$,
$[\alpha_3]\cong [\alpha_5]\cong [2]\times H_4$, and
$[\alpha_4]\cong [2]\times [3]\times [3]$.

\subsection{$E_7$}
The Dynkin diagram is obtained from that of $E_6$ by adding $\alpha_7$ adjacent to $\alpha_6$. Then
$[\alpha_3]\cong [2]\times H_{5}$, $[\alpha_4]\cong [2]\times
[3]\times [4]$, $[\alpha_5]\cong [3]\times H_{4}$, $[\alpha_6]\cong
[2]\times J^2([2]\times [3])$, $[\alpha_7]\cong J^3([2]\times [3])$.

\subsection{$E_8$}
The Dynkin diagram is obtained from that of $E_7$ by adding $\alpha_8$ adjacent to $\alpha_7$. Then
$[\alpha_3]\cong [2]\times H_6$, $[\alpha_4]\cong [2]\times
[3]\times [5]$, $[\alpha_5]\cong [4]\times H_4$, $[\alpha_6]\cong
[3]\times J^2([2]\times [3])$ and $[\alpha_7]\cong [2]\times
J^3([2]\times [3])$.

\subsection{Exceptions to the pattern \eqref{pattern}}
There are seven such exceptions: $[\alpha_4]$ in $F_4$
(extra-special); $[\alpha_2]$ in $E_6$ (extra-special); $[\alpha_1]$
in $E_7$ (extra-special), $[\alpha_2]$ in $E_7$;  $[\alpha_1]$,
$[\alpha_2]$ and $[\alpha_8]$ in $E_8$. We present the Hasse
diagrams for two of them in Fig.~3 and Fig.~4. Note that each $\alpha_i$ is
an ending point in the Dynkin diagram.

\begin{figure}[]
\centering \scalebox{0.8}{\includegraphics{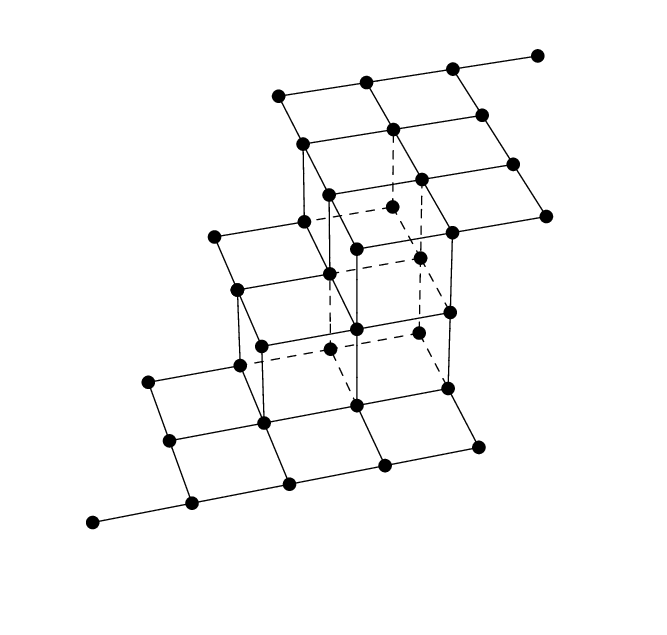}} \caption{The
Hasse diagram of $[\alpha_2](E_7)$}
\end{figure}

\begin{figure}[]
\centering \scalebox{0.8}{\includegraphics{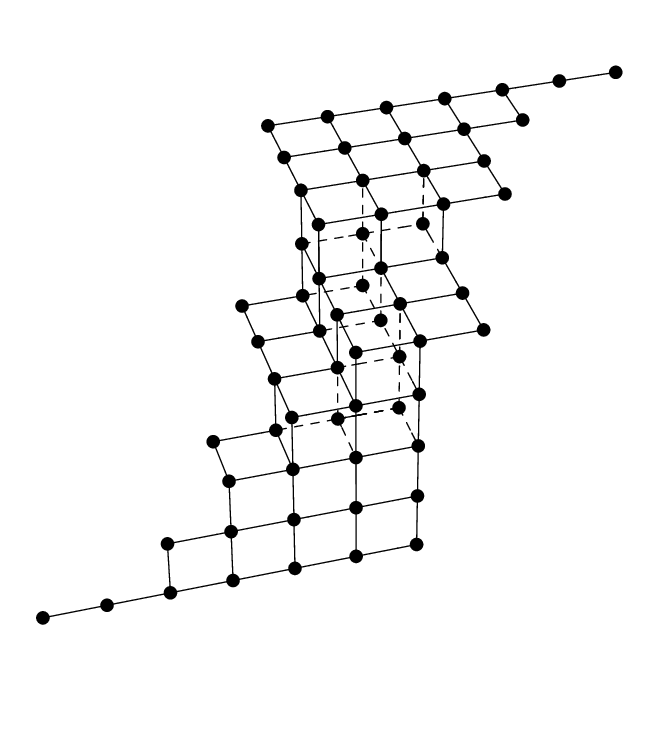}} \caption{The
Hasse diagram of $[\alpha_1](E_8)$}
\end{figure}

\section{A proof of Panyushev's $\caM$-polynomial conjecture}

In this section, we will prove Theorem \ref{thm-M-poly-main}. Note
that if \eqref{pattern} holds for $[\alpha_i]$, namely,
$[\alpha_i]\cong [k]\times P$ for some minuscule poset $P$, then $P$
is Gaussian by Theorem \ref{thm-Proctor}. Thus Remark
\ref{rmk-thm-Proctor} allows us to conclude that $[k] \times P$ is
pleasant, as desired. This finishes the proof of Theorem
\ref{thm-M-poly-main} for those $[\alpha_i]$ bearing the pattern
\eqref{pattern}. Since the extra-special cases have been handled in \cite{P}, it remains to check the four non-extra-special
posets in Section 4.10.

\subsection{$E_7$}
Using \texttt{Mathematica}, one can verify that
$$
\caM_{[\alpha_2]}(t)=\frac{(1-t^8)(1-t^{10})(1-t^{11})(1-t^{12})(1-t^{14})}{(1-t)(1-t^3) (1-t^4)
(1-t^5)(1-t^7)}.
$$
Thus $[\alpha_2]$ is pleasant.

\begin{rmk}
The RHS of \eqref{KM} for $[\alpha_2]\times [6]$ is not a polynomial since the RHS of \eqref{number-antichains} for it is not an integer.
Thus $[\alpha_2]\times [6]$ is not pleasant, and $[\alpha_2]$ is \emph{not} Gaussian.
\end{rmk}

\subsection{$E_8$}
Using \texttt{Mathematica}, one can verify that
\begin{align*}
\caM_{[\alpha_1]}(t)&=\frac{(1 - t^{14}) (1 - t^{17}) (1 - t^{18}) (1 - t^{20}) (1 - t^{23})}{(1 - t) (1 -t^4) (1 - t^6) (1 - t^7) (1 - t^{10})};  \\
\caM_{[\alpha_2]}(t)&=\frac{(1 - t^{11}) (1 - t^{12}) (1 - t^{13}) (1 - t^{14}) (1 - t^{15}) (1 -
   t^{17})}{(1 - t) (1 - t^3) (1 - t^4) (1 - t^5) (1 - t^6) (1 - t^7)}; \\
\caM_{[\alpha_8]}(t)&= \frac{(1 - t^{20}) (1 - t^{24}) (1 - t^{29})}{(1 - t) (1 - t^6) (1 - t^{10})}.
\end{align*}
Thus every poset is pleasant, and the $E_8$ case is finished.

\begin{rmk}
Similar to the previous remark, one can show that none of $[\alpha_1]$,
$[\alpha_2]$, $[\alpha_8]$ is Gaussian.
Moreover, none of the three extra-special posets in Section 4.10 is
Gaussian.
\end{rmk}

\section{The number $\caM_{\Delta(1)}(-1)$}

This section is devoted to proving Theorems \ref{thm-M-1-main} and
\ref{thm-fixed-point-main}. We continue to let $\frg$ be a
finite-dimensional simple Lie algebra over $\bbC$, and label the
simple roots as in Section 4. Let $w_0$ be the longest element of
the Weyl group $W=W(\frg, \frh)$ of $\frg$. The following result is
well-known.

\begin{lemma}\label{lemma-w0}
\begin{itemize}
\item[a)] $w_0=-1$ in $A_1$, $B_n$, $C_n$, $D_{2n}$, $E_7$, $E_8$, $F_4$ and $G_2$.
\item[b)] In $A_{n-1}$, $w_0(\alpha_i)=-\alpha_{n-i}$, $1\leq i\leq n-1$.
\item[c)] In $D_{2n-1}$, $-w_0$ interchanges $\alpha_{2n-2}$ and $\alpha_{2n-1}$,  while preserves other simple roots.
\item[d)] In $E_{6}$, $-w_0$ interchanges $\alpha_1$ and $\alpha_6$, $\alpha_3$ and $\alpha_5$, while preserves
$\alpha_2$ and $\alpha_4$.
\end{itemize}
\end{lemma}

Let $\varpi_i$ be a minuscule fundamental weight, and let
$P_{\varpi_i}$ be the corresponding minuscule poset. Recall that
$w_0\in W$ acts as an order-reversing involution on the weight poset
$\Lambda_{\varpi_i}\cong J(P_{\varpi_i})$. As on p.~479 of
\cite{Stem}, this involution transfers to an order-reversing
involution on the poset $P_{\varpi_i}$. Indeed, for every $x\in P_{\varpi_i}$,
the lower order ideals $I_{\leq x}=\{y\in P_{\varpi_i}\mid y\leq x\}$ and
$I_{< x}=\{y\in P_{\varpi_i}\mid y< x\}$ have the property that
$I_{< x}^{c}- I_{\leq x}^{c}=\{x^{\prime}\}$ for some $x^{\prime}\in P_{\varpi_i}$.
Then one can easily check that $x\mapsto x^{\prime}$ is indeed an order-reversing involution
on $P_{\varpi_i}$.
 We denote the corresponding
complemented poset by $(P_{\varpi_i}, c)$. That is, $J(P_{\varpi_i},
c)$ and $(\Lambda_{\varpi_i}, w_0)$ are isomorphic as complemented posets.
Now let us recall Theorem 4.1 of \cite{Stem}.

\begin{thm}\label{thm-Stembridge} \emph{(\textbf{Stembridge})}
Let $(P_{\varpi_i}, c)$ be a complemented minuscule poset as above. Then $\caM_{[m]\times P_{\varpi_i}}(-1)$ is the number
of self-complementary lower ideals of $[m]\times P_{\varpi_i}$, or
equivalently, the number of multi-chains $I_m\subseteq \cdots
\subseteq I_1$ ($I_j\in J(P_{\varpi_i})$) such that $I_j^c=I_{m+1-j}$ (see
Lemma \ref{lemma-ideals-CnP}). Here recall that $I_j^c:=P_{\varpi_i}\setminus \{c(x)\mid x\in I_j\}$.
\end{thm}

The following lemma gives the order-reversing involution
$c$ on $P_{\varpi_i}$ explicitly.

\begin{lemma}\label{lemma-uniform-transfer}
In the setting of \eqref{minu-real}, denote by $w_0^i$ the longest element of the Weyl group
of $\frg(0)$ in the $1$-standard $\bbZ$-grading such that
$\Pi(1)=\{\alpha_i\}$. Then we have
\begin{equation}\label{minu-transfer}
(\Lambda_{\varpi_i}, w_0)\cong J([\alpha_i^{\vee}], w_0^i).
\end{equation}
\end{lemma}
\begin{proof}

Since we need to
pass to Lie subalgebras of $\frg$ frequently, let us explicitly
give the types to avoid confusion. For instance,
$[\alpha_k](A_{n})$ means the $[\alpha_k]$ in $\frg$ of type
$A_{n}$.

Firstly, note that $[\alpha_k](A_{n})\cong [k]\times [n+1-k]$. Moreover, the order-reversing involution induced by $w_0^k(A_n)$ on $[k]\times [n+1-k]$ is the one sending $(i, j)$ to $(k+1-i, n+2-k-j)$. By Example 4.2 of \cite{Stem}, we have
\begin{equation}\label{transfer-A}
(\Lambda_{\varpi_k}(A_n), w_0(A_n))\cong J([\alpha_k](A_{n}),
w_0^k(A_{n})).
\end{equation}

Note that $\Lambda_{\varpi_n}(B_n)\cong J(H_n)$, see Example 3.3 of \cite{Stem}.
On the other hand, $[\alpha_n](C_n)\cong H_n$ has only one order-reversing
inclusion $(i, j)\mapsto (n+1-j, n+1-i)$, see Example 4.3 of \cite{Stem}.
Thus we must have \begin{equation}\label{transfer-B}
(\Lambda_{\varpi_n}(B_{n}), w_0(B_n)) \cong J([\alpha_n](C_n), w_0^n(C_{n})).
\end{equation}
Similarly, one sees that \eqref{minu-transfer} holds for $(\Lambda_{\varpi_{n-1}}, w_0(D_n))$
and $(\Lambda_{\varpi_{n}}, w_0(D_n))$.

Note that $[\alpha_1](B_n)\cong [2n-1]$ has a unique order-reversing
involution. Thus we must have
\begin{equation}\label{transfer-C} (\Lambda_{\varpi_1}(C_{n}),
w_0(C_n)) \cong J([\alpha_1](B_n), w_0^1(B_{n})).
\end{equation}

Now let us prove that for $n\geq 4$, we have
\begin{equation}\label{transfer-D}
(\Lambda_{\varpi_1}(D_n), w_0(D_n))\cong J([\alpha_1](D_n),
w_0^1(D_{n})).
\end{equation}
Note first that $[\alpha_1](D_n)\cong K_{n-2}$ and
$J(K_{n-2})=K_{n-1}$. Moreover, $K_n$ has exactly two
order-reversing involutions, one has two fixed points, while the
other has none. Now let us proceed according to two cases.
\begin{itemize}
\item[(i)] $n$ is even. Then the $w_0^1(D_{n})$ action on $[\alpha_1](D_n)$ has
two fixed points, while the complemented poset $J([\alpha_1](D_n),
w_0^1(D_{n-1}))$ has none. On the other hand, the first fundamental
weight in $D_{n}$ is $e_1$. Since $w_0(D_{n})=-1$, one sees that
$(\Lambda_{\varpi_1}(D_n), w_0(D_n))$ has no fixed point as well.
Thus \eqref{transfer-D} holds. Here in the special case that $n=4$,
we interpret $D_3$ as $A_3$.
\item[(ii)] $n$ is odd. Then the $w_0^1(D_{n})$ action on $[\alpha_1](D_n)$ has
no fixed point, while the complemented poset $J([\alpha_1](D_n),
w_0^1(D_{n}))$ has two. On the other hand, using Lemma
\ref{lemma-w0}(c), one sees that $(\Lambda_{\varpi_1}(D_n),
w_0(D_n))$ also has two fixed points. Thus \eqref{transfer-D} holds.
\end{itemize}
To sum up, \eqref{transfer-D} is always true.

Now let us prove that
\begin{equation}\label{transfer-E6}
(\Lambda_{\varpi_1}(E_6), w_0(E_6)) \cong J([\alpha_1](E_6),
w_0^1(E_6)).
\end{equation}
Note that on one hand in the graded poset $\Lambda_{\varpi_1}(E_6)$ (see the right one of Fig.~2), the middle level consists of three elements with rank $9$:
$$
s_1 s_3 s_4 s_5 s_2 s_4 s_3 s_1(\varpi_1), s_3 s_4 s_2 s_6 s_5 s_4 s_3 s_1(\varpi_1), s_5 s_4 s_2 s_6 s_5 s_4 s_3 s_1(\varpi_1).
$$
All of them are fixed by $w_0(E_6)$. On the other hand, one can check that there are three lower ideals
in $([\alpha_1](E_6), w_0^1(E_6))$ with size $8$, and they are all fixed points in $J([\alpha_1](E_6), w_0^1(E_6))$. Then \eqref{transfer-E6} follows directly.

Finally, we mention that
\begin{equation}\label{transfer-E7}
(\Lambda_{\varpi_1}(E_7), w_0(E_7))\cong J([\alpha_1](E_7),
w_0^1(E_7)).
\end{equation}
We note that Fig.~1 (right) of \cite{Pr} gives the structure $\Lambda_{\varpi_1}(E_7)$, based on which
one can figure out the structure of $(\Lambda_{\varpi_1}(E_7), w_0(E_7))$. In particular there is a unique cube. On the other hand, recall that $[\alpha_1](E_7)\cong J^3([2]\times [3])$. Then one can determine the structure of $([\alpha_1](E_7), w_0^1(E_7))$. Passing to $J([\alpha_1](E_7), w_0^1(E_7))$, we will also get a unique cube.  By matching the patterns around  the two cubes, one will obtain \eqref{transfer-E7}. We omit the details.
\end{proof}

Note that $([\alpha_1](A_n), w_0^1(A_{n}))$ is just the poset $([n],
\leq)$ equipped with the order-reversing involution $j\mapsto
n+1-j$. For simplicity, sometimes we just denote the complemented
minuscule poset $([\alpha_1](A_n), w_0^1(A_{n}))$ by $[n]$ instead. Now we are ready to prove Theorem \ref{thm-M-1-main}.

\medskip
\noindent \emph{Proof of Theorem \ref{thm-M-1-main}.} Firstly, let us handle those $[\alpha_i]$ bearing the pattern \eqref{pattern}.
Similar to Section 4, our discussion is case-by-case, yet the method is the same.

By Lemma \ref{lemma-uniform-transfer} and Theorem \ref{thm-Stembridge}, if the fundamental weight $\varpi_i$ is minuscule,  then Theorem \ref{thm-M-1-main} holds for $[\alpha_i^{\vee}]$ in $(\Delta^{\vee})^+$.

Let us investigate $[\alpha_k](D_n)$ for $2\leq
k\leq n-3$ in details. Since now $$w_0^k(D_n)=w_0(A_{k-1}) w_0(D_{n-k})$$ (again
$A_3$ is viewed as $D_3$) and the two factors commute, we have that
$$([\alpha_k](D_{n}), w_0^k(D_n))\cong [k]\times ([\alpha_1](D_{n-k+1}), w_0^1(D_{n-k+1})).$$
Now by applying \eqref{transfer-D} to $D_{n-k+1}$ and using Theorem \ref{thm-Stembridge}, one sees that Theorem  \ref{thm-M-1-main} holds for $[\alpha_k](D_n)$. For some other cases, we list the substitutes for \eqref{transfer-D} as follows:
\begin{itemize}
\item[$\bullet$] $[\alpha_k](B_n)$, use $J([\alpha_1](B_n), w_0^1(B_{n}))\cong (\Lambda_{\varpi_1}(A_{2n-1}), w_0(A_{2n-1}))$;
\item[$\bullet$] $[\alpha_k](C_n)$, use $J([\alpha_1](C_n), w_0^1(C_{n}))\cong (\Lambda_{\varpi_1}(A_{2n-2}), w_0(A_{2n-2}))$;
\item[$\bullet$] $[\alpha_i]$ where $\alpha_i$ is a branching point, use \eqref{transfer-A};
\item[$\bullet$]  $[\alpha_6](E_7)$, use \eqref{transfer-E6};
\item[$\bullet$]  $[\alpha_7](E_8)$, use \eqref{transfer-E7}.
\end{itemize}

Secondly, we have used \texttt{Mathematica} to check Theorem \ref{thm-M-1-main} for the seven posets in Section 4.10. \hfill\qed

Before proving Theorem \ref{thm-fixed-point-main}, we prepare the following.

\begin{lemma}\label{lemma-fixed-pt}
If the $w_0^i$ action on $[\alpha_i](\frg)$ has fixed point(s), then
$([\alpha_i], w_0^i)$ has no self-complementary lower ideal.
\end{lemma}
\begin{proof}
Let $I$ be any lower ideal of $[\alpha_i]$. Note that $I$ is
self-complementary if and only if $[\alpha_i]\setminus I$ is
self-complementary. Indeed,
\begin{align*}
I = [\alpha_i]\setminus w_0^i(I) &\Leftrightarrow I \sqcup
w_0^i(I)=[\alpha_i]\\
&\Leftrightarrow  ([\alpha_i]\setminus I) \sqcup
w_0^i([\alpha_i]\setminus I)=[\alpha_i]  \\
&\Leftrightarrow  [\alpha_i]\setminus I = [\alpha_i]\setminus
w_0^i([\alpha_i]\setminus I).
\end{align*}
Now suppose that there exists $\gamma\in[\alpha_i]$ such that
$w_0^i\gamma=\gamma$. If there was a self-complementary lower ideal $I$ of
$[\alpha_i]$, let us deduce contradiction. Indeed, if $\gamma\in I$, then $\gamma=w_0^i(\gamma)\in w_0^{i}(I)$. This contradicts to the assumption that $I$ is self-complementary. If $\gamma\in [\alpha_i]\setminus I$,
then one would also get a contradiction since the upper ideal $[\alpha_i]\setminus I$ is
self-complementary as well.
\end{proof}

Finally, let us deduce Theorem \ref{thm-fixed-point-main}.

\medskip
\noindent \emph{Proof of Theorem \ref{thm-fixed-point-main}.} If the
$w_0^i$ action on $[\alpha_i]$ has fixed point(s), then
$([\alpha_i], w_0^i)$ has no self-complementary lower ideal by
Lemma \ref{lemma-fixed-pt}. Thus $\caM_{[\alpha_i]}(-1)=0$ by
Theorem \ref{thm-M-1-main}.

Conversely, if $\caM_{[\alpha_i]}(-1)=0$, then we need to exhibit
the fixed points of the $w_0^i$ action on $[\alpha_i]$. We note that
$\caM_{[\alpha_i]}(-1)=0$ in the following cases:
\begin{itemize}
\item[$\bullet$]
$[\alpha_{i}](A_{2n+1})$ for those odd $i$ between $1$ and $2n+1$;
\item[$\bullet$] $[\alpha_{i}](B_{n})$ for $i$ odd;
\item[$\bullet$] $[\alpha_{n}](C_{n})$; $[\alpha_{n-1}](D_{n})$ and $[\alpha_{n}](D_{n})$;
\item[$\bullet$] $[\alpha_{i}](D_{2n})$ for those odd $i$ between $1$ and $2n-3$;
\item[$\bullet$] $[\alpha_2](E_7)$ and $[\alpha_7](E_7)$.
\end{itemize}
In the classical types, aided
by Lemma \ref{lemma-w0}, one can identify the fixed points easily.
We provide the fixed points for the last two cases.
\begin{align*}
[\alpha_2](E_7): [1, 1, 1, 2, 1, 1, 0], [1, 1, 2, 2, 1, 0, 0], [0,
1, 1, 2, 1, 1, 1]; \\
 [\alpha_7](E_7): [1, 1, 2, 2, 1, 1, 1], [1, 1, 1,
2, 2, 1, 1], [0, 1, 1, 2, 2, 2, 1],
\end{align*}
where the roots are expressed in terms of the simple ones.
\hfill\qed

\section{A proof of Panyushev's $\caN$-polynomial conjecture}

This section is devoted to proving Theorem \ref{thm-N-poly-main}.
Note that (a) trivially implies (b) since the constant term of any
$\caN$ polynomial is always $1$, while (c) is just a restatement of
(b). Since $\Delta(1)$ is Sperner by Lemma 2.6 of \cite{P} and each
rank level $\Delta(1)_i$ is an antichain, one sees that (c) implies
(d). Therefore, it remains to show that (d) implies (a). This will
be carried out in the remaining part of this section. Firstly, let us prepare the following.

\begin{lemma}\label{lemma-N-poly-known} We have that
\begin{itemize}
\item[a)] $\caN_{[n]\times [m]}(t)=\sum_{i\geq 0}{n \choose i}{m\choose
i}t^i$. The poset $[n]\times [m]$ has a unique rank level of maximal
size if and only if $m=n$, if and only if $\caN_{[n]\times [m]}(t)$
is palindromic.
\item[b)] $\caN_{H_n}(t)=\sum_{i\geq 0}{n+1 \choose 2i}t^i$.
The poset $H_n$ has a unique rank level of maximal size if and only
if $n$ is odd, if and only if $\caN_{H_n}(t)$ is palindromic.

\item[c)] The following posets have at least two rank levels of maximal size: $[2]\times H_4$, $[2]\times
[3]\times [3]$; $[2]\times H_5$, $[2]\times
J^2([2]\times [3])$; $[2]\times H_6$,
$[4]\times H_4$, $[2]\times [3]\times [5]$, $[3]\times J^2([2]\times
[3])$, $[2]\times J^3([2]\times [3])$.

\item[d)] The following posets have a unique rank level of maximal size: $[2]\times
[3]\times [4]$, $[3]\times H_4$. Moreover, their $\caN$ polynomials
are palindromic.
\end{itemize}
\end{lemma}
\begin{proof}
Part (a) follows directly from Lemma \ref{lemma-antichain-CnP}. See
also item 1 on p.~1201 of \cite{P}. Part (b) is item 2
on p.~1201 of \cite{P}. One easily verifies part (c) and the first statement of part (d). For the
second statement of (d), we mention that
\begin{align*}
\caN_{[2]\times
[3]\times [4]}(t) &=1+ 24t + 120 t^2+ 200 t^3 + 120 t^4 + 24 t^5 +t^6,\\
\caN_{[3]\times H_4}(t) &=1+ 30t + 165 t^2+ 280 t^3 + 165 t^4 + 30
t^5 +t^6.
\end{align*}

\end{proof}

Now let us investigate the $\caN$-polynomial of $[m]\times K_n$.
Suppose now we have exactly one ball labeled $i$ for $1\leq i \neq
n+1 \leq 2n+1$, and two \emph{distinct} balls labeled $n+1$. We want
to put them into $m$ boxes arranged from left to right so that there
is at most one ball in each box with the only exception that the two
balls labeled $n+1$ can be put in the same box, and that the
relative order among the labels $1,2, \dots, 2n+1$ under $\leq$ are
preserved when we read them off the balls from left to right. Let us
denote by $A_{n, m}(i)$ the number of filling $i$ balls into the
boxes so that the above requirements are met. By Lemma
\ref{lemma-antichain-CnP}, one sees easily that
\begin{equation}\label{Anmi}
\caN_{[m]\times K_n}(t)=\sum_{i=0}^{m+1} A_{n, m}(i) t^i.
\end{equation}

\begin{thm}\label{thm-N-poly-CmKn} The following are equivalent:
\begin{itemize}
\item[a)]
The poset $[m]\times K_n$ has a unique rank level of maximal size;
\item[b)] $m=1$ or $2n+1$;
\item[c)] $\caN_{[m]\times K_n}$ is monic;
\item[d)] $\caN_{[m]\times K_n}$ is palindromic.
\end{itemize}
\end{thm}
\begin{proof}
The equivalence between (a) and (b) is elementary, while that between (b) and (c) follows from \eqref{Anmi}. Part (d) trivially implies (c). Now it remains to show that (b) implies (d). When $m$=1, we have
$$
\caN_{K_n}(t)=1+(2n+2)t+t^2,
$$
which is palindromic. Now let us show that $\caN_{[2n+1]\times K_n}$
is palindromic.

To obtain $A_{n, 2n+1}(1)$, we note there are two possibilities:
neither of the two balls labeled $n+1$ is chosen; exactly one of the
two balls labeled $n+1$ is  chosen. This gives
$$
A_{n, 2n+1}(1)={2n\choose 1}{2n+1\choose 1}+{2\choose 1}{2n+1\choose 1}.
$$
To obtain $A_{n, 2n+1}(2n+1)$, we note there are three
possibilities: exactly one of the two balls labeled $n+1$ is
chosen; both of the two balls labeled $n+1$ are  chosen and they are
put in the same box; both of the two balls labeled $n+1$ are chosen
and they are put in different boxes. This gives
$$
A_{n, 2n+1}(2n+1)={2\choose 1}{2n+1\choose 2n+1}+{2n \choose 2n-1}{2n+1\choose 2n}+2{2n \choose 2n-1}{2n+1\choose 2n+1}.
$$
One sees that $A_{n, 2n+1}(1)=A_{n, 2n+1}(2n+1)$.

Now let $2\leq i\leq 2n$. To obtain $A_{n, 2n+1}(i)$, we note there
are four possibilities: neither of the two balls labeled $n+1$ is
chosen; exactly one of the two balls labeled $n+1$ is chosen; both
of the two balls labeled $n+1$ are  chosen and they are put in the
same box; both of the two balls labeled $n+1$ are chosen and they
are put in different boxes. Therefore $A_{n, 2n+1}(i)$ is equal to
$$
{2n\choose i}{2n+1\choose i}+{2\choose 1}{2n \choose i-1}{2n+1\choose i}+{2n \choose i-2}{2n+1\choose i-1}+2{2n \choose i-2}{2n+1\choose i}.
$$
Thus
$$
A_{n, 2n+1}(i)={2n \choose i-2}{2n+1\choose i-1}+{2n\choose i}{2n+1\choose i}+2{2n+1 \choose i-1}{2n+1\choose i}.
$$
Substituting $i$ by $2n+2-i$ in the above formula gives
$$
A_{n, 2n+1}(2n+2-i)={2n \choose i}{2n+1\choose i}+{2n\choose i-2}{2n+1\choose i-1}+2{2n+1 \choose i}{2n+1\choose i-1}.
$$
Thus $A_{n, 2n+1}(i)=A_{n, 2n+1}(2n+2-i)$ for $2\leq i\leq 2n$.

To sum up, we have shown that $\caN_{[2n+1]\times K_n}$ is palindromic. This finishes the proof.

\end{proof}

Now we are ready to prove Panyushev's $\caN$-polynomial
conjecture.

\noindent \emph{Proof of Theorem \ref{thm-N-poly-main}.}
 As noted at the beginning of this section, it remains to show that if
$\Delta(1)$ has a unique rank level of maximal size, then
$\caN_{\Delta(1)}(t)$ is palindromic. By Lemma \ref{lemma-N-poly-known}
and Theorem \ref{thm-N-poly-CmKn}, Theorem \ref{thm-N-poly-main} holds for those non-abelian or non-extra-special $[\alpha_i]$'s
bearing the pattern \eqref{pattern}. Now it remains to handle the
four non-extra-special posets in Section 4.10.

For $E_7$, it remains to consider $[\alpha_2]$. Indeed, it has a unique rank
level of maximal size. Moreover, using \texttt{Mathematica}, we
obtain that
$$
\caN_{[\alpha_2]}(t) =1+ 35t + 140 t^2+ 140 t^3 + 35 t^4 + t^5,
$$
which is palindromic.

For $E_8$, it remains to consider $[\alpha_1]$, $[\alpha_2]$, and $[\alpha_8]$. Indeed, each
of them has more than one rank level of maximal size. This finishes the proof.
\hfill\qed

\begin{rmk}
We mention that in $E_8$, one has
\begin{align*}
\caN_{[\alpha_1]}(t)&=1+ 64t + 364 t^2+ 520 t^3 + 208 t^4 + 16
t^5,\\
\caN_{[\alpha_2]}(t)&=1+ 56t + 420 t^2+ 952 t^3 + 770 t^4 + 216 t^5
+16 t^6,\\
\caN_{[\alpha_8]}(t)&=1+ 56t + 133 t^2+ 42 t^3.
\end{align*}
\end{rmk}

\section{Structure of the Panyushev orbits}

This section is devoted to investigating the structure of the Panyushev orbits.
To be more precise, we shall establish Theorems \ref{thm-main-reverse-operator-orbit} and \ref{thm-main-cyclic-sieving}.

\noindent \emph{Proof of Theorem \ref{thm-main-reverse-operator-orbit}.}
We keep the notation of Section 2. In particular, $L_i$'s are the full rank lower ideals of $P$,
and $(I_1, I_2)$, where $I_i\in J(P)$ and $I_2\subseteq I_1$, stands for a general lower ideal of $[2]\times P$. Recall that $\mathfrak{X}$ acts on lower ideals as well.

Note that when
$\frg$ is $A_n$,  the extra-special $\Delta(1)\cong [n-1]\sqcup
[n-1]$; when $\frg$ is $C_n$, the extra-special $\Delta(1)\cong
[2n-2]$. One can verify Theorem \ref{thm-main-reverse-operator-orbit} for these two cases
without much effort. We omit the details.

For $\frg=B_n$,  the extra-special $\Delta(1)= [2]\times [2n-3]$.
Now $|\Pi_{l}|=n-1$, $h-1=2n-1$, and $h^*-2=2n-3$. As in Section 2,
let $L_i$ ($0\leq i\leq 2n-3$) be the rank level lower ideals. For
simplicity, we denote $\mathfrak{X}_{[2]\times [2n-3]}$ by
$\mathfrak{X}$. For any $0\leq i\leq n-2$, let us analyze the type I
$\mathfrak{X}$-orbit $\mathcal{O}(L_i, L_i)$ via the aid of Lemma
\ref{lemma-operator-ideals-CmP}:
\begin{align*}
\mathfrak{X}(L_i, L_i)&=(L_{i+1}, L_0),\\
\mathfrak{X}^{2n-4-i}(L_{i+1}, L_0)&=(L_{2n-3}, L_{2n-4-i}),\\
\mathfrak{X}(L_{2n-3}, L_{2n-4-i})&=(L_{2n-3-i}, L_{2n-3-i}),\\
\mathfrak{X}(L_{2n-3-i}, L_{2n-3-i})&=(L_{2n-2-i}, L_{0}),\\
\mathfrak{X}^{i-1}(L_{2n-2-i}, L_{0})&=(L_{2n-3}, L_{i-1}),\\
\mathfrak{X}(L_{2n-3}, L_{i-1})&=(L_{i}, L_{i}).
\end{align*}
Thus $\mathcal{O}(L_i, L_i)$ consists of $2n-1$ elements. Moreover,
in this orbit,  $(L_{2n-2-\frac{i+1}{2}}, L_{\frac{i-1}{2}})$ (resp.
$(L_{n+\frac{i}{2}-1}, L_{n-\frac{i}{2}-2})$) is the unique lower
ideal with size $2n-3$ when $i$ is odd (resp. even). Since there are
$(n-1)(2n-1)$ lower ideals in $[2]\times [2n-3]$ by Corollary
\ref{cor-thm-M-poly-main}, one sees that all the
$\mathfrak{X}$-orbits have been exhausted, and Theorem
\ref{thm-main-reverse-operator-orbit} holds for $B_{n}$.

Let us consider $D_{n+2}$, where the extra-special $\Delta(1)\cong
[2]\times K_{n-1}$. We adopt the notation as in Section 2. For simplicity,
we write $\mathfrak{X}_{[2]\times K_{n-1}}$ by $\mathfrak{X}$.  We propose
the following.

\textbf{Claim.} $\mathcal{O}(L_i, L_i)$, $0\leq i\leq n-1$,
 $\mathcal{O}(I_n, I_n)$, and $\mathcal{O}(I_{n^{\prime}},
I_{n^{\prime}})$ exhaust the orbits of $\mathfrak{X}$ on $[2]\times
K_{n-1}$. Moreover, each orbit has size $2n+1$ and contains a unique
lower ideal with size $2n$.

Indeed, firstly, for any $0\leq i\leq n-1$, observe that by Lemma
\ref{lemma-operator-ideals-CmP}, we have
\begin{align*}
\mathfrak{X}(L_i, L_i)&=(L_{i+1}, I_0),\\
\mathfrak{X}^{2n-i-2}(L_{i+1}, L_0)&=(L_{2n-1}, L_{2n-i-2}),\\
\mathfrak{X}(L_{2n-1}, L_{2n-i-2})&=(L_{2n-i-1}, L_{2n-i-1}),\\
\mathfrak{X}(L_{2n-i-1}, L_{2n-i-1})&=(L_{2n-i}, L_{0}),\\
\mathfrak{X}^{i-1}(L_{2n-i}, L_{0})&=(L_{2n-1}, L_{i-1}),\\
\mathfrak{X}(L_{2n-1}, L_{i-1})&=(L_{i}, L_{i}).
\end{align*}
Thus the type I orbit $\mathcal{O}(L_i, L_i)$ consists of $2n+1$
elements. Moreover, in this orbit,  $(L_{2n-i+\frac{i-1}{2}},
L_{\frac{i-1}{2}})$ (resp. $(L_{n+\frac{i}{2}},
L_{n-\frac{i}{2}-1})$) is the unique lower ideal with size $2n$ when
$i$ is odd (resp.  even).

Secondly, assume that $n$ is even and let us analyze the orbit
$\mathcal{O}(I_n, I_n)$. Indeed, by Lemma \ref{lemma-operator-ideals-CmK}, we have
\begin{align*}
\mathfrak{X}(I_n, I_n)&=(I_{n^{\prime}}, L_0),\\
\mathfrak{X}^{n-1}(I_{n^{\prime}}, L_0)&=(I_{n}, L_{n-1}),\\
\mathfrak{X}(I_{n}, L_{n-1})&=(L_{n}, I_{n}),\\
\mathfrak{X}^{n-1}(L_{n}, I_{n})&=(L_{2n-1}, I_{n^{\prime}}),\\
\mathfrak{X}(L_{2n-1}, I_{n^{\prime}})&=(I_{n}, I_{n}).
\end{align*}
Thus the type II orbit $\mathcal{O}(I_n, I_n)$ consists of $2n+1$ elements. Moreover,
in this orbit,  $(I_n, I_n)$ is the unique ideal with size $2n$. The
analysis of the orbit $\mathcal{O}(I_{n^{\prime}}, I_{n^{\prime}})$
is entirely similar.

Finally, assume that $n$ is odd and let us analyze the orbit
$\mathcal{O}(I_n, I_n)$. Indeed,  by Lemma \ref{lemma-operator-ideals-CmK},  we have
\begin{align*}
\mathfrak{X}(I_n, I_n)&=(I_{n^{\prime}}, L_0),\\
\mathfrak{X}^{n-1}(I_{n^{\prime}}, L_0)&=(I_{n^{\prime}}, L_{n-1}),\\
\mathfrak{X}(I_{n^{\prime}}, L_{n-1})&=(L_{n}, I_{n^{\prime}}),\\
\mathfrak{X}^{n-1}(L_{n}, I_{n^{\prime}})&=(L_{2n-1}, I_{n^{\prime}}),\\
\mathfrak{X}(L_{2n-1}, I_{n^{\prime}})&=(I_{n}, I_{n}).
\end{align*}
Thus the type II  orbit $\mathcal{O}(I_n, I_n)$ consists of $2n+1$ elements. Moreover,
in this orbit,  $(I_n, I_n)$ is the unique ideal with size $2n$. The
analysis of the orbit $\mathcal{O}(I_{n^{\prime}}, I_{n^{\prime}})$
is entirely similar.

To sum up, we have verified the claim since there are $(n+2)(2n+1)$
lower ideals in $[2]\times K_{n-1}$ by Corollary
\ref{cor-thm-M-poly-main}. Note that $|\Pi_{l}|=n+2$, $h=h^*=2n+2$
for $\frg=D_{n+2}$, one sees that Theorem
\ref{thm-main-reverse-operator-orbit} holds for $D_{n+2}$.

Theorem \ref{thm-main-reverse-operator-orbit} has been verified for all exceptional Lie
algebras using \texttt{Mathematica}. We only present the details for
$E_6$, where $\Delta(1)=[\alpha_2]$. Note that $|\Pi_l|=6$,
$h-1=11$, $h^*-2=10$. On the other hand, $\mathfrak{X}$ has six
orbits on $\Delta(1)$, each has $11$ elements. Moreover, the size of
the lower ideals in each orbit is distributed as follows:
\begin{itemize}
\item[$\bullet$] $0, 1, 2, 4, 7, \textbf{10}, 13, 16, 18, 19, 20$;

\item[$\bullet$] $3, 4, 5, 6, 9, \textbf{10}, 11, 14, 15, 16, 17$;

\item[$\bullet$] $3, 4, 5, 6, 9, \textbf{10}, 11, 14, 15, 16, 17$;

\item[$\bullet$] $7, 7, 8, 8, 9, \textbf{10}, 11, 12, 12, 13, 13$;

\item[$\bullet$] $5, 6, 6, 8, 9, \textbf{10}, 11, 12, 14, 14, 15$;

\item[$\bullet$] $7, 7, 8, 8, 9, \textbf{10}, 11, 12, 12, 13, 13$.
\end{itemize}
One sees that each orbit has a unique Lagrangian lower ideal. \hfill\qed

\medskip

\noindent \emph{Proof of Theorem \ref{thm-main-cyclic-sieving}.}
Based on the structure of $\Delta(1)$ for $1$-standard $\bbZ$-gradings of $\frg$ in Section 3,
the desired CSP for the triple $(\Delta(1)$, $\caM_{\Delta(1)}(t)$, $\langle\mathfrak{X}_{\Delta(1)}\rangle)$ follows from Theorems 1.1, 1.2 and 10.1 of \cite{RS} combined with Table 1 at the root of the paper, where $n$ is the order of $\mathfrak{X}_{\Delta(1)}$, $[n]_t:= \frac{1-t^n}{1-t}$ and \eqref{CSP-2} has been applied.
\hfill\qed

\medskip

\begin{center}
\begin{tabular}{|c|c|c|c|}
\multicolumn{4}{c} {Table 1. Information for the Panyushev orbits}\\ \hline
$\Delta(1)$  &  $n$ & orbits (size $\times$ number) &  $\caM_{\Delta(1)}(t) \mod (t^n-1)$ \\ \hline
$[\alpha_4](F_4)$  & 11&  $11 \times 2$ & $2\times [11]_{t}$ \\ \hline
$[\alpha_2](E_6)$  &  11& $11 \times 6$ & $6 \times [11]_{t}$ \\ \hline
$[\alpha_1](E_7)$  &  17& $17 \times 7$ & $7 \times [16]_{t}$ \\
$[\alpha_2](E_7)$  &  14& $14 \times 25+ 2\times 1$ & $1+t^7+ 25 \times [14]_{t}$ \\
$[\alpha_5](E_7)$  & 10&  $10 \times 67+ 2\times 1$ & $1+t^5+ 67 \times [10]_{t}$ \\ \hline
$[\alpha_1](E_8)$  & 23&  $23 \times 51$ & $51 \times [23]_{t}$ \\
$[\alpha_2](E_8)$  &   17& $17 \times 143$ & $143  \times[17]_{t}$ \\
$[\alpha_5](E_8)$  &  11& $11 \times 252$ & $252 \times [11]_{t}$ \\
$[\alpha_8](E_8)$  &  29& $29 \times 8$ & $8 \times [29]_{t}$ \\ \hline
\end{tabular}
\end{center}

\medskip

\centerline{\scshape Acknowledgements}  We thank Dr.~Bai, Dr.~Wang, and Prof.~Stembridge for helpful discussions. A revision of the paper was carried out during Dong's visit of MIT. He thanks the math department there sincerely for offering excellent working conditions. Finally, we express our sincere gratitude to the referees for giving us valuable suggestions.

\end{document}